\documentclass{patmorin}
\usepackage[utf8]{inputenc}
\usepackage{microtype}
\usepackage{amsthm,amsmath,graphicx}
\usepackage{pat}
\usepackage[letterpaper]{hyperref}
\usepackage[table,dvipsnames]{xcolor}
\definecolor{linkblue}{named}{Blue}
\hypersetup{colorlinks=true, linkcolor=linkblue,  anchorcolor=linkblue,
citecolor=linkblue, filecolor=linkblue, menucolor=linkblue,
urlcolor=linkblue} 
\usepackage{wasysym}

\title{\MakeUppercase{Anagram-Free Chromatic Number is not Pathwidth-Bounded}%
   \thanks{This work was partly funded by NSERC and the Ontario Ministry of
    Research, Innovation and Science}}

\author{Paz Carmi%
   \thanks{Department of Computer Science,
           Ben-Gurion University of the Negev}\quad%
   Vida Dujmović%
   \thanks{School of Computer Science and Electrical Engineering, 
           University of Ottawa}\quad%
   Pat Morin%
   \thanks{School of Computer Science, Carleton University}}%

\DeclareMathOperator{\pw}{pw}

\begin{document}
\maketitle
\begin{abstract}
  The anagram-free chromatic number is a new graph parameter
  introduced independently by Kamčev, Łuczak, and Sudakov
  \cite{kamcev.luczak.ea:anagram-free} and Wilson and Wood
  \cite{wilson.wood:anagram-free}.  In this note, we show that there
  are planar graphs of pathwidth 3 with arbitrarily large anagram-free
  chromatic number.  More specifically, we describe $2n$-vertex planar
  graphs of pathwidth 3 with anagram-free chromatic number $\Omega(\log n)$.
  We also describe $kn$ vertex graphs with pathwidth $2k-1$ having
  anagram-free chromatic number in $\Omega(k\log n)$.
\end{abstract}
%
%

\section{Introduction}
\pagenumbering{arabic}

A string $s=s_1,\ldots,s_{2k}$ is called an \emph{anagram} if
$s_1,\ldots,s_k$ is a permutation of $s_{k+1},\ldots,s_{2k}$.
For a graph $G$, a $c$-colouring $\varphi:V(G)\to\{1,\ldots,c\}$ is
\emph{anagram-free} if, for every odd-length path $v_1,v_2,\ldots,v_{2k}$
in $G$, the string $\varphi(v_1),\ldots,\varphi(v_{2k})$ is not an
anagram.  The \emph{anagram-free chromatic number} of $G$, denoted
$\pi_\alpha(G)$, is the smallest value of $c$ for which $G$ has an
anagram-free $c$-colouring.

Answering a long-standing question of Erd\H{o}s and Brown, Ker\"anen
\cite{keranen:abelian} showed that the path $P_n$ on $n$ vertices has an
anagram-free 4-colouring.  A straightforward divide-and-conquer
algorithm applied to any $n$-vertex graph of treewidth $k$ yields
an anagram-free $O(k\log n)$-colouring.  The same divide-and-conquer
algoritm, applied to graphs that exclude a fixed minor gives an anagram
free $O(\sqrt{n})$-colouring \cite{kamcev.luczak.ea:anagram-free}.
An interesting variant of this divide-and-conquer algorithm is used by
Wilson and Wood \cite{wilson.wood:anagram-free} to obtain anagram-free
$(4k+1)$-colourings of trees of pathwidth $k$.  On the negative side,
Kamčev, Łuczak, and Sudakov \cite{kamcev.luczak.ea:anagram-free} and
Wilson and Wood \cite{wilson.wood:anagram-free} have shown that there
are trees---even binary trees---with arbitrarily large anagram-free
chromatic number.  These results, and some others, are summarized in
\tabref{results}.

\begin{table}
  \begin{center}
    \begin{tabular}{lll}
      \textbf{Graph class} & \textbf{Bounds} & \textbf{Reference} \\ \hline
       Paths & $\pi_\alpha(G)= 4$ & \cite[Theorem~1]{keranen:abelian}  \\
       Graphs of treewidth $k$ & $\pi_\alpha(G)\in O(k\log n)$ & folklore  \\
       Graphs excluding a minor of size $h$ & $\pi_\alpha(G)\in O(h^{3/2}n^{1/2})$ 
              & \cite[Proposition~1.2]{kamcev.luczak.ea:anagram-free} \\
       Trees & $\pi_\alpha(G)\in\Omega(\log n/\log\log n)$ 
              & \cite[Theorem~3]{wilson.wood:anagram-free} \\
       Trees of pathwidth $k$ & $k\le \pi_\alpha(G)\le 4k+1$ 
              & \cite[Theorem~5]{wilson.wood:anagram-free} \\
       Trees of radius $r$ & $r\le \pi_\alpha(G)\le r+1$ 
              & \cite[Theorem~4]{wilson.wood:anagram-free} \\
       Binary trees & $\pi_\alpha(G)\in\Omega(\sqrt{\log n/\log\log n})$ 
              & \cite[Proposition~1.1]{kamcev.luczak.ea:anagram-free} \\ 
       $4$-regular graphs & $\pi_\alpha(G)\in \Omega(\sqrt{n}/\log n)$ 
           & \cite[Proposition~3.1]{kamcev.luczak.ea:anagram-free} \\
       $d$-regular graphs & $\pi_\alpha(G)\in \Omega(n)$ & \cite[Theorem~1.3]{kamcev.luczak.ea:anagram-free} \\
       Subdivisions of graphs & $\pi_\alpha(G) \le 8$ 
         & \cite[Theorem~6]{wilson.wood:anagram-free2} \\
       Planar graphs & $\pi_\alpha(G)\in O(\sqrt{n})$ &
              \cite[Corollary~2.3]{kamcev.luczak.ea:anagram-free} \\
       Planar graphs of maximum degree 3 & $\pi_\alpha(G)\in\Omega(\log n/\log\log n)$ 
         & \cite[Proposition~2.4]{kamcev.luczak.ea:anagram-free} \\
           & & \cite[Theorem~1]{wilson.wood:anagram-free} \\
       Planar graphs of pathwidth $3$ & $\pi_\alpha(G) \in \Omega(\log n)$
         & \thmref{main} \\
       Graphs of pathwidth $k>3$ & $\pi_\alpha(G) \in \Omega(k\log n)$
         & \thmref{main-2} 
    \end{tabular}
  \end{center}
  \caption{Bounds on anagram-free chromatic number.  Upper bounds apply
  to all graphs in the class. Lower bounds apply to some graphs in
  the class.}
  \tablabel{results}
\end{table}

All of the examples of graphs having large anagram-free chromatic
number are graphs with large pathwidth \cite{robertson.seymour:graph}.
Therefore, an obvious question is whether anagram-free chromatic number
is pathwidth-bounded, i.e., can $\pi_\alpha(G)$ be upper bounded by some
function of the pathwidth $\pw(G)$ of $G$?  Such a result seems plausible,
for two reasons:
\begin{enumerate}
  \item pathwidth is a measure of how path-like a graph is and Ker\"anen
     showed that paths have anagram-free 4-colourings; and
  \item the result of Wilson and Wood \cite{wilson.wood:anagram-free}
     shows that $\pi_\alpha(T)\le 4\pw(T)+1$ for every tree, $T$.
\end{enumerate}
The purpose of this note, however, is to show that the result of Wilson
and Wood can not be strengthened even to planar graphs of pathwidth 3
and maximum degree 5. (Here and throughout, $\log x=\log_2 x$ denotes the binary logarithm of $x$.)

\begin{thm}\thmlabel{main}
  For every $n\in\N$, there exists a $2n$-vertex planar graph of
  pathwidth 3 and maximum degree 5 whose anagram-free chromatic number
  is at least $\log(n+1)$.
\end{thm}

\begin{thm}\thmlabel{main-2}
  For every $n\in\N$ and every integer $k\ge 3$, there exists a
  $kn$-vertex graph of pathwidth $2k-1$ and maximum degree $3k-1$
  whose anagram-free chromatic number is at least $(k-2)\log(n/3)$.
\end{thm}

These two results show that the straightforward divide-and-conquer
algorithm using separators gives asymptotically worst-case optimal
colourings for graphs of pathwidth $k$ and graphs of treewidth $k$.

\section{Proof of \thmref{main}}

Let $s\in\Sigma^*$ be a string over some alphabet $\Sigma$.  For each
$a\in\Sigma$, we let $n_a(s)$ denote the number of occurences of $a$
in $s$.  We say that $s$ is \emph{even} if $n_a(s)$ is even for each
$a\in\Sigma$.  The following lemma says that strings with no even
substrings must use an alphabet of at least logarithmic size.

\begin{lem}\lemlabel{parity}
  If $s=s_0,\ldots,s_{2n-1}\in\Sigma^{2n}$ and $|\Sigma|< \log(n+1)$,
  then $s$ contains a non-empty even substring $s_{2i},\ldots,s_{2j-1}$
  for some $0\le i < j\le n$.
\end{lem}

\begin{proof}
  For any string $q\in\Sigma^*$, we define the \emph{parity vector}
  $P(q)=\langle n_a(q)\bmod 2: a\in\Sigma \rangle$ and observe that $q$ is
  even if and only if $P(q)=\langle 0,\ldots,0\rangle$.   Furthermore, for
  two strings $p$ and $q$, the parity vector of their concatenation $pq$
  is equal to the xor-sum (i.e., modulo 2 sum) of their parity vectors:
  \[
     P(pq) = P(p)\oplus P(q) \enspace .
  \] 
  Define the strings $t_0,\ldots,t_n$, where $t_0$ is the empty string
  and, for each $i\in\{1,\ldots,n\}$, define $t_i=s_0,\ldots,s_{2i-1}$.

  Now consider the parity vectors $P(t_0),P(t_1),\ldots,P(t_n)$.
  Each of these $n+1$ vectors is a binary string of length
  $|\Sigma| < \log(n+1)$ therefore, there must exist two indices
  $i,j\in\{0,\ldots,n\}$ with $i<j$ such that $P(t_i)=P(t_j)$.  However,
  \[
      P(t_j) = P(t_i) \oplus P(s_{2i},\ldots,s_{2j-1}) 
  \]
  and since $P(t_i)=P(t_j)$, this implies that $P(s_{2i},\ldots,s_{2j-1})=\langle0,\ldots,0\rangle$
  and $s_{2i},\ldots,s_{2j-1}$ is even, as required.
\end{proof}

The next lemma says that if we split an even string into consecutive
pairs, then we can can colour one element of each pair red and the other
blue in such a way that the resulting red and blue multisets are exactly
the same.
\begin{lem}\lemlabel{splitting}
  Let $s=s_0,\ldots,s_{2r-1}\in\Sigma^{2r}$ be an even string. Then there
  exists a binary sequence $v_0,\ldots,v_{r-1}$ such that the string
  $s_v=s_{0+v_0},s_{2+v_1},\ldots,s_{2(r-1)+v_{r-1}}$ 
  has $n_a(s_v)=n_a(s)/2$
  for all $a\in\Sigma$.
\end{lem}

\begin{proof}
  Suppose for the sake of contradiction that the lemma is not true, and
  let $s$ be the shortest counterexample.  For $v\in\{0,1\}^{r}$, let
  $s_{\overline{v}}=s_{0+1-v_0},s_{2+1-v_1},\ldots,s_{2(r-1)+1-v_{r-1}}$
  be the complement of $s_v$.  Let $v\in\{0,1\}^{r}$ be the binary vector
  that minimizes
  \begin{equation}\eqlabel{minimizer}
     \sum_{a\in\Sigma}|n_a(s_v)-n_a(s_{\overline{v}})| \enspace .
  \end{equation}
  Since $s$ is a counterexample to the lemma, \eqref{minimizer} is
  greater than zero.

  For each $j\in\{0,\ldots,r-1\}$, let $x_{j}=s_{2j+v_j}$ and
  let $y_j=s_{2j+1-v_j}$ so that $s_v=x_0,\ldots,x_{r-1}$ and
  $s_{\overline{v}}=y_0,\ldots,y_{r-1}$.  Since \eqref{minimizer}
  is non-zero, there exists some $j_1$ such that $n_{x_{j_1}}(s_v) >
  n_{x_{j_1}}(s_{\overline{v}})$.  This means that $n_{y_{j_1}}(s_v) \ge
  n_{y_{j_1}}(s_{\overline{v}})$, otherwise flipping\footnote{Here and
  throughout, flipping a binary variable $b$ means changing its value to $1-b$.}
  $v_{j_1}$ would decrease \eqref{minimizer} by two.  Furthermore,
  $y_{j_1}\neq x_{j_1}$ since, otherwise, we could remove $s_{2j}$
  and $s_{2j+1}$ from $s$ and obtain a smaller counterexample, since
  the value of $v_j$ has no effect on \eqref{minimizer}.

  Refer to \figref{swapping}.
  Let $a_1=x_{j_1}$ and for $k=2,3,4\ldots$, define $a_k = y_{j_{k-1}}$
  and define $j_k$ to be any index such that $x_{j_k}=a_k$.  Notice that
  that $n_{a_k}(s_v)\ge n_{a_{k}}(s_{\overline{v}})$ since, otherwise,
  flipping $v_{j_1},\ldots,v_{j_{k-1}}$ would decrease the value of
  \eqref{minimizer}.  Indeed, flipping $v_{j_1},\ldots,v_{j_{k-1}}$
  decreases $n_{a_1}(s_v)$ by one, increases $n_{a_k}(s_v)$ by one, and
  does not change $n_a(s_v)$ for any $a\in\Sigma \setminus \{a_1,a_k\}$.
  This implies that $j_{k}$ is well-defined since
  $n_{a_k}(s_v)\ge n_{a_k}(s_{\overline{v}})\ge 1$.

  \begin{figure}
    \begin{center}
       \includegraphics{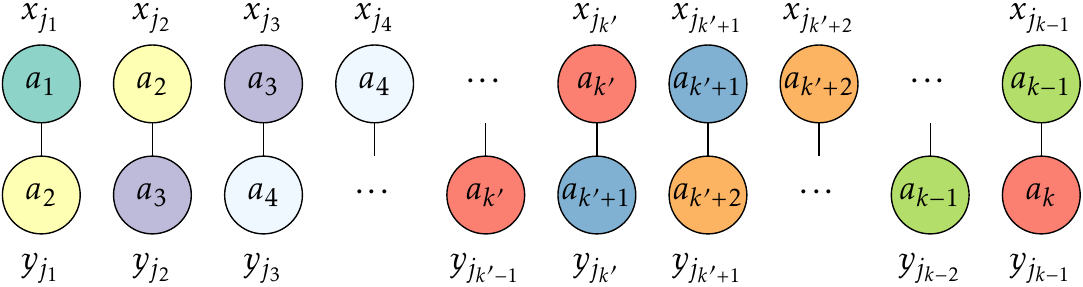}
    \end{center}
    \caption{The proof of \lemref{splitting}.}
    \figlabel{swapping}
  \end{figure}

  Since $s$ is finite, there is some minimum value $k$ such that
  $a_{k}=a_{k'}$ for some $k' < k$.  This defines a sequence of indices
  $j_{k'},\ldots,j_{k-1}$ such that
  \begin{enumerate}
     \item $a_{k'}=x_{j_{k'}}=y_{j_{k-1}}=a_{k}$; 
     \item $a_{\ell}=y_{j_{\ell-1}}=x_{j_\ell}$ for all $\ell\in\{k'+1,\ldots,k-1\}$.
  \end{enumerate}
  In words, for each $\ell\in\{k',\ldots,k\}$, each occurrence of $a_\ell$ in $s_v$ is matched with a corresponding
  occurrence of $a_\ell$ in $s_{\overline{v}}$.
  We claim that this contradicts
  the minimality of $s$. Indeed, by removing
  $s_{2_{j_{k'}}},s_{2j_{k'}+1},s_{2j_{k'+1}},s_{2j_{k'+1}+1},\ldots,s_{2j_{k-1}},s_{2j_{k-1}+1}$
  from $s$ we obtain a smaller counterexample.
\end{proof}

\begin{proof}[Proof of \thmref{main}]
  Let $G$ be the graph with vertex set
  $V(G)=\{x_1,\ldots,x_{n},y_1,\ldots,y_{n}\}$ and with edge set
  \[
    E(G) = \{x_iy_i : i\in\{0,\ldots,n-1\}\} 
              \cup \bigcup_{i=0}^{n-2} \{x_iy_{i+1},x_{i+1}y_i\} \enspace .
  \]
  The graph $G$ has pathwidth 3 as can be seen from the path decomposition
  $B_0,\ldots,B_{n-2}$ where $B_i=\{v_i,w_i,v_{i+1},w_{i+1}\}$.  See
  \figref{graph}. Although not immediately obvious from \figref{graph},
  $G$ is also planar---see \figref{planar}.
  
  \begin{figure}
    \begin{center}
      \includegraphics{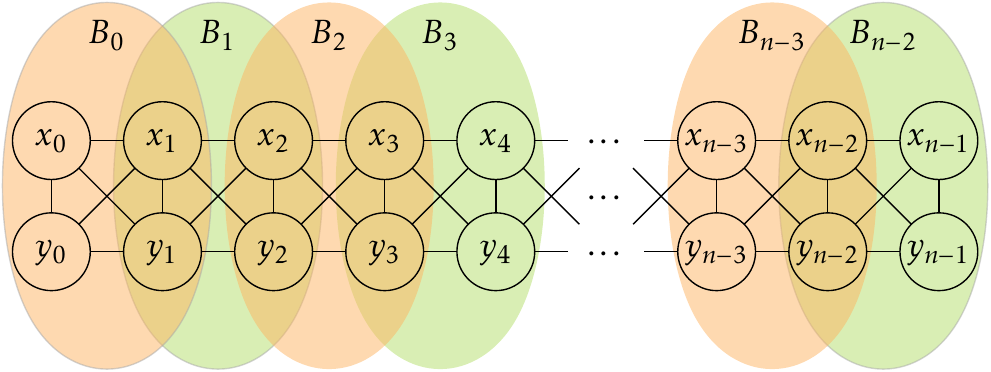}
    \end{center}
    \caption{The graph $G$ in the proof of \thmref{main}.}
    \figlabel{graph}
  \end{figure}
  
  \begin{figure}
    \begin{center}
      \includegraphics[width=\textwidth]{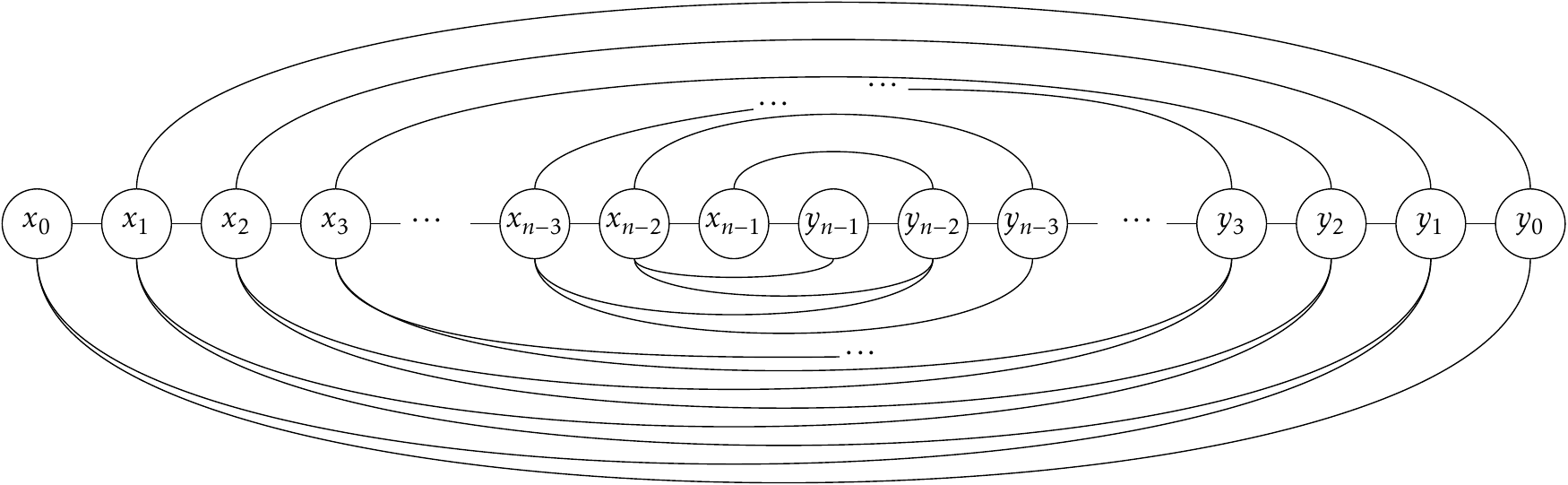}
    \end{center}
    \caption{The graph $G$ in the proof of \thmref{main} is planar
     and is even a 2-page graph.}
    \figlabel{planar}
  \end{figure}
  
  Now, consider some colouring $\varphi:V(G)\to\Sigma$
  with $|\Sigma| < \log(n+1)$.  Applying \lemref{parity} to
  the string 
  $s=\varphi(x_1),\varphi(y_1),\ldots,\varphi(x_n),\varphi(y_n)$ we
  conclude that there is some $i < j$ such that 
  $\varphi(x_{i}),\varphi(y_i),\ldots,\varphi(x_j),\varphi(y_j)$
  is even.  By \lemref{splitting} and the
  symmetry between each $x_i$ and $y_i$ we can assume that
  $n_a(\varphi_(x_i),\ldots,\varphi(x_j))=n_a(\varphi(y_i),\ldots,\varphi(y_j))$
  for each $a\in\Sigma$.  But then the path
  $x_i,\ldots,x_j,y_j,y_{j-1},\ldots,y_i$ has a colour sequence
  $\varphi(x_i),\ldots,\varphi(x_j),\varphi(y_j),\varphi(y_{j-1}),\ldots,\varphi(y_i)$
  that is an anagram.
\end{proof}

\section{Proof of \thmref{main-2}}

\begin{lem}\lemlabel{marriage}
   For every sequence of sets $X_1,\ldots,X_n\subseteq\Sigma$, each of
   size $k>2$, with $|\Sigma|< (k-2)\log(n/3)$, there exists indices
   $1\le i< j\le n$ and subsets $X_i',\ldots,X_j'$ such that, for each
   $\ell\in\{i,\ldots,j\}$,  $X_\ell'\subseteq X_\ell$, $|X_\ell'|\ge 2$
   and, for each $a\in\Sigma$ the number of subsets in $X_i',\ldots,X_j'$
   that contain $a$ is even.
\end{lem}

\begin{proof}
   For any $1\le i\le j\le n$, let $\Sigma_{i,j}=\bigcup_{\ell=i}^j
   X_i$ and, for any $I\subset\Sigma_{i,j}$, let
   $N_{i,j}(I)=\{\ell\in\{i,\ldots,j\}: X_\ell\cap I\neq\emptyset \}$.
   We distinguish between two cases.

   \paragraph{Case 1:}
   There is some pair of indices $1\le i\le j\le n$ such that, for every
   $I\subseteq\Sigma_{i,j}$, 
   \begin{equation}
       |N_{i,j}(I)| \ge |I|/(k-2) \enspace . \eqlabel{hall}
   \end{equation}
   In this case
   we will show the existence of the desired sets $X_{i}',\ldots,X_{j}'$.
   Without loss of generality, assume $i=1$, $j=n$, and define $N=N_{1,n}$.

   Define a bipartite graph $H$ with vertex
   set $V(H)=\Sigma\cup\{1,\ldots,n\}$ and edge set $E(H)=\{(a,i):
   i\in\{1,\ldots,n\}, a\in X_i\}$.  We will show that $E(H)$ contains a
   subset $E'$ such that each element $a\in\Sigma$ appears exactly once in
   $E'$ and each element of $\{1,\ldots,n\}$ appears at most $k-2$ times
   in $E'$.  That is, $E'$ defines a mapping $f:\Sigma\to\{1,\ldots,n\}$
   in which, for any $i\in\{1,\ldots,n\}$, $|f^{-1}(i)|\le k-2$.

   The existence of the mapping $f$ establishes the lemma since we can
   start with $X_i'=X_i$ for all $i\in\{1,\ldots,n\}$ and then, for each
   $a\in\Sigma$ that appears an odd number of times, we can remove $a$
   from the set $X_{f(a)}'$.  When this process is complete each $X_i'$
   has size at least 2 and each $a\in\Sigma$ occurs in an even number
   of the sets $X_1',\ldots,X_n'$.

   All that remains is to prove the existence of the edge set $E'$, which
   we do using an augmenting paths argument like that used, for example,
   to prove Hall's Marriage Theorem.  Consider an edge set $E'\subseteq
   E(H)$ that contains exactly one edge incident to each $a\in\Sigma$
   and let $f:\Sigma\to\{1,\ldots,n\}$ be the corresponding mapping.
   Then we define
   \[
      \Phi(E') = \sum_{i=1}^n\max\{0, |f^{-1}(i)|-(k-2)\} \enspace .
   \]
   Note that the set $E'$ we hope to find has $\Phi(E')=0$.  Now, select
   some $E'$ that minimizes $\Phi(E')$.  If $\Phi(E')=0$ then we are done,
   so assume by way of contradiction, that $\Phi(E') > 0$.  Thus, there
   exists some index $i_0\in\{1,\ldots,n\}$ such that $|f^{-1}(i_0)|\ge k-1$
   and therefore the set $\Sigma_0=f^{-1}(i_0)$ has size at least $k-1$.
   Therefore, 
   \[
       |N(\Sigma_0)| \ge \left\lceil\frac{|\Sigma_0|}{k-2}\right\rceil \ge 
\left\lceil\frac{k-1}{k-2}\right\rceil = 2 \enspace .
   \]
   In particular, $N(\Sigma_0)\setminus\{i_0\}$ is non-empty.  Let
   $I_0=\{i_0\}$ and observe that each $i_1\in N(\Sigma_0)\setminus I_0$
   must have $|f^{-1}(i_1)|\ge k-2$ since, otherwise we could replace
   the edge $(a_1,i_0)$ with $(a_1,i_1)$ in $E'$ and this would decrease
   $\Phi(E')$.  Let $I_1=N(\Sigma_0)$ and let $\Sigma_1=\bigcup_{i_1\in
   I_i} f^{-1}(i_1)$.  We have just argued that
   \[
        |\Sigma_1|\ge |I_1|(k-2)+1
   \]
   and therefore, 
   \[
          |N(\Sigma_1)| \ge \left\lceil\frac{|\Sigma_1|}{k-2}\right\rceil
            \ge \left\lceil\frac{|I_1|(k-2)+1}{k-2}\right\rceil
            \ge |I_1|+1 \enspace .
   \]
   But now we can continue this argument, defining $I_j=N(\Sigma_{j-1})$
   and $\Sigma_j=\bigcup_{i_j\in I_j} f^{-1}(i_j)$.  Again,
   each $i_j\in I_j\setminus \bigcup_{\ell=1}^{j-1} I_\ell$
   must have $|f^{-1}(i_j)|\ge k-2$, otherwise we can find a path
   $i_0,a_0,i_1,a_1,\ldots,a_{j-1}i_j$ and replace, in $E'$, the edges
   $i_0a_0,\ldots,i_{j-1}a_{j-1}$ with $a_0i_1,a_1i_2,\ldots,a_{j-1}i_j$
   which would decrease $\Phi(E')$.  In this way, we obtain an infinite
   sequence of subsets $I_0,\ldots,I_\infty\subseteq \{1,\ldots,n\}$
   such that $|I_{j}|>|I_{j-1}|$.  This is clearly a contradiction,
   since each $|I_j|$ is an an integer in $\{1,\ldots,n\}$.

   \paragraph{Case 2:}
   For every $1\le i< j\le n$, there exists a set $I\subset\Sigma_{i,j}$
   such that $|N_{i,j}(I)| < |I|/(k-2)$.  In this case, we will show
   that $|\Sigma|\ge (k-2)\log(n/3)$.

   Before jumping into the messy details, we sketch an inductive proof
   that gives the main intuition for why $|\Sigma|\in \Omega(k\log n)$:
   There is some set $I_0\subset\Sigma$ such that $N(I_0)$ partitions
   $\{1,\ldots,n\}$ into $O(|I|/k)$ intervals.  One such interval
   $i_0,\ldots,j_0$ must have size $\Omega(nk/|I|)$.  By induction on $n$,
   $|\Sigma_{i_0,j_0}| = \Omega(k\log (nk/|I|))$.  But $\Sigma_{i_0,j_0}$
   is disjoint from $I$, so 
   \[  |\Sigma| \ge |I| + \Omega(k\log (nk/|I|))
        = |I| + \Omega(k\log n) - O(k\log(|I|/k))
        = \Omega(k\log n) \enspace .
   \]
   The messy details occur when $|I|=k-1$ since then the $|I|$ and
   $-O(k\log(|I|/k)$ terms are close in magnitude.

   Let $n_0=n$, $i_0=1$, $j_0=n$, $\Sigma_0=\Sigma$ and let
   $I_0\subseteq\Sigma_0$ be such that $|N(I_0)| < |I|/(k-2)$.
   For each integer $\ell$ with $n_{\ell-1}\ge 1$, we define 
   \begin{enumerate}
     \item $i_\ell$ and $j_\ell$ such that 
       $i_{\ell-1}\le i_\ell < j_\ell\le j_{\ell-1}$, 
       $\{i_\ell,\ldots,j_\ell\}\cap N_{i_{\ell-1},j_{\ell-1}}(I_{\ell-1})=\emptyset$, 
       and $n_\ell=j_\ell-i_\ell+1$ is maximized.
     \item $I_\ell\subset\Sigma_{i_\ell,j_\ell}$ such that
      $|N_{i_\ell,j_\ell}(I_\ell)| < |I_\ell|/k$;
   \end{enumerate}
   In words, $N_{i_{\ell-1},j_{\ell-1}}(I_{\ell-1})$ partitions $i_{\ell-1},\ldots,j_{\ell-1}$
   into intervals and we choose $i_\ell$ and $j_\ell$ to be the endpoints
   of a largest such interval and recurse on that interval using a new
   set $I_\ell$. Letting $y_\ell=|N_{i_\ell,j_\ell}(I_\ell)|$, observe that, for $\ell \ge 1$,
   \[
        n_\ell \ge \frac{n_{\ell-1}-y_{\ell-1}}{y_{\ell-1}+1}
               > \frac{n_{\ell-1}}{y_{\ell-1}+1} - 1 \enspace .
   \]
   By expanding the preceding equation we can easily show that
   \[
           n_{\ell} \ge \frac{n}{\prod_{\tau=0}^{\ell-1}(y_\tau+1)} - 2 \enspace .
   \]
   Note that $n_{\ell+1}$ is defined until $n_{\ell} < 1$ so combining this
   with the preceding equation and taking logs yields
   \begin{equation}
           \sum_{\tau=0}^{\ell-1}(y_\tau+1) > \log(n/3) \eqlabel{constraint}
   \end{equation}
   Finally, observe that the sets $I_0,\ldots,I_{\ell-1}$ are disjoint, so
   \begin{equation}
          |\Sigma| \ge \sum_{\tau=0}^{\ell-1} |I_\tau| >
             \sum_{\tau=0}^{\ell-1} (k-2)y_\tau \enspace . \eqlabel{objective}
   \end{equation}
   Now, minimizing \eqref{objective} subject to \eqref{constraint} and using
   the fact that each $y_\tau \ge 1$ is an integer shows that
   $|\Sigma|\ge (k-2)\log (n/3)$, as desired.  (The minimum is obtained when
   $\ell=\log(n/3)$ and $y_1=y_2=\cdots=y_{\ell-1}=1$.)
\end{proof}

\begin{proof}[Proof of \thmref{main-2}]
The pathwidth $2k-1$ graph, $G$, used in this proof is a natural
generalization of the pathwidth $3$ graph used in the proof of
\thmref{main}.  The $kn$ vertices of $G$ are partitioned in subsets
$V_1,\ldots,V_n$, each size of size $k$.  For each $i\in\{1,\ldots,n\}$,
$V_i$ is a clique and, for each $i\in\{1,\ldots,n-1\}$, every
vertex in $V_i$ is adjacent to every vertex in $V_{i+1}$.  That this
graph has pathwidth $2k-1$ can be seen from the path decomposition
$B_1,\ldots,B_{n-2}$ where each $B_i=\{V_i\cup V_{i+1}\}$.

Suppose we have some colouring $\varphi:V(G)\to\Sigma$, with
$|\Sigma|<(k-2)\log(n/3)$.  Define the sets $X_1,\ldots,X_n$ where
$X_i=\{\varphi(v): v\in V_i\}$.  By \lemref{marriage}, we can find indices
$i\in\{0,\ldots,n-1\}$ and $r>0$ and subsets $V_1',\ldots,V_{r}'$ such
that, for each $\ell\in\{1,\ldots,r\}$,  $V_\ell'\subseteq V_{i+\ell}$, $|V_\ell'|\ge 2$,
and such that each colour $a\in\Sigma$ appears in an even number of
$V_1',\ldots,V_r'$.

Next, label the vertices in $V_1',\ldots,V_r'$ red and blue as
follows.  If $|V_i'|$ is even, then label half its vertices red and
half its vertices blue, arbitrarily.  Let $Q_1,\ldots,Q_t$ denote the
subsequence of $V_1',\ldots,V_r'$ consisting of only sets of odd size
(so the vertices in $Q_1,\ldots,Q_t$ are not labelled red or blue yet).
Then, for odd values of $i$, label $\ceil{|Q_i|/2}$ vertices
of $Q_i$ red and the remaining blue.  For even values of $i$ label
$\floor{|Q_i|/2}$ vertices of $Q_i$ red and the remaining blue.
Observe that, since $\sum_{i=1}^r|V_i'|$ is even, $t$ is also even,
so exactly half the vertices in $\bigcup_{i=1}^r V_i'$ are red and half
are blue.

Now, consider the following perfect bichromatic matching of
$\bigcup_{i=1}^r V_i'$: In every set $V_i'$ of even size we match each
red vertex in $V_i'$ with a blue vertex in $V_i'$.  In each odd size set
$Q_i$, we match $\lfloor|Q_i|/2\rfloor$ red vertices with blue vertices
leaving one vertex $v_i$ unmatched.  This leaves $t$ unmatched vertices
$v_1,\ldots,v_t$ and these vertices alternate colour between red and
blue. To complete the matching, we match $v_{2i}$ with $v_{2i-1}$ for
each $i\in\{1,\ldots,t/2\}$.

Now, treat this matching as a long string $s=x_1,y_1,\ldots,x_q,y_q$
where each $x_i=\varphi(v_i)$, each $y_i=\varphi(w_i)$, and each
$(v_i,w_i)$ is a matched pair of vertices.  Now, applying
\lemref{splitting} to $s$, we obtain two sets of vertices
$V=\{v_1',\ldots,v_q'\}$ and $W=\{w_1',\ldots,w_q'\}$ such that, 
for each $a\in\Sigma$,
$n_a(\varphi(v_1),\ldots,\varphi(v_q))=n_a(\varphi(w_1),\ldots,\varphi(w_q))$.
Thus, all that remains is to show that $G$ contains a path $P$ whose first
half is some permutation of $V$ and whose second half
is some permutation of $W$.  But this is obvious, because,
for each $i\in\{1,\ldots,r\}$, $V_i'$ contains at least one vertex of $V$
and at least one vertex of $W$.  Thus, the path $P$ first visits all the
vertices of $V\cap V_1'$ followed by all the vertices of $V\cap V_2'$,
and so on until visiting all the vertices in $V\cap V_r'$. Next, the path
returns and visits all the vertices in $W\cap V_r'$, $W\cap V_{r-1}'$,
and so on back to $W\cap V_1'$.  The existence of the path $P$ shows
that no colouring of $G$ with fewer than $(k-2)\log(n/3)$ colours is anagram-free, so $\pi_\alpha(G) \ge (k-2)\log(n/3)$.
\end{proof}

\section{Remarks}

We have show that anagram-free chromatic number is not pathwidth-bounded,
even for planar graphs.  The graph we use in the proof of \thmref{main}
is a 2-page graph; it has a book embedding using two pages.  Outerplanar
graphs have a book embedding using a single page.  Is anagram-free
chromatic number pathwidth-bounded for outerplanar graphs?  We do not even know if the $2\times n$ grid has constant anagram-free chromatic number.

\bibliographystyle{plainurl}
\bibliography{anagram}

\end{document}